\documentclass[runningheads]{llncs}
\usepackage{graphicx}
\usepackage{hyperref}
\usepackage{amsmath, amsfonts}
\usepackage{algorithm}
\usepackage[noend]{algpseudocode}

\usepackage{siunitx}

\usepackage{tikz}
\usepackage{pgfplots}
\pgfplotsset{compat=1.15}
\usetikzlibrary{external} 
\tikzexternaldisable

\algrenewcommand{\algorithmiccomment}[1]{$\triangleright$ \emph{#1}}
\newcommand{\diam}[1]{{\operatorname{diam}\left({#1}\right)}}
\newcommand{\dist}[2]{\operatorname{dist}\left({#1},{#2}\right)}
\DeclareMathOperator{\dir}{dir}
\DeclareMathOperator{\child}{child}
\DeclareMathOperator{\parent}{parent}
\DeclareMathOperator{\level}{level}

\DeclareMathOperator{\diag}{diag}

\newcommand{\new}[1]{#1}

\begin{document}
\title{Complexity Analysis of a Fast Directional Matrix-Vector Multiplication} 
\author{G\"unther Of\orcidID{0000-0003-2258-7001} \and Raphael Watschinger\orcidID{0000-0002-7750-8561} }
\authorrunning{G.~Of and R.~Watschinger}
\institute{Graz University of Technology, Institute of Applied Mathematics,
	Steyrergasse 30, 8010 Graz, Austria, \email{of@tugraz.at}, \email{watschinger@math.tugraz.at}}
\maketitle              
\begin{abstract}
  We consider a fast, data-sparse directional method to realize matrix-vector products related to point evaluations of the Helmholtz kernel. The method is based on a hierarchical partitioning of the point sets and the matrix. The considered directional multi-level approximation of the Helmholtz kernel can be applied even on high-frequency levels efficiently. We provide a detailed analysis of the almost linear asymptotic complexity of the presented method. Our numerical experiments are in good agreement with the provided theory.
\keywords{Helmholtz \and Fast multipole method \and Hierarchical matrix.}
\end{abstract}
\section{Introduction}
In this paper we consider an efficient method for the computation of the matrix-vector product for a fully populated matrix $A \in \mathbb{C}^{N_T \times N_S}$ with entries
\begin{align} 
	A[j,k] &= f(x_j, y_k), \label{ow:eq_helmholtz_matrix} \\
	f(x,y) &= \frac{\exp (i \kappa |x-y|)}{4 \pi |x-y|} \nonumber
\end{align}
where $f$ is the Helmholtz kernel, $\kappa > 0$ the wave number and $P_T=\{x_j\}_{j=1}^{N_T}$ and $P_S=\{y_k\}_{k=1}^{N_S}$ are two sets of points  in $\mathbb{R}^3$. Similar matrices arise in the solution of boundary value problems for the Helmholtz equation by boundary element methods. Using standard matrix-vector multiplication is prohibitive for large $N_T$ and~$N_S$ due to the asymptotic runtime and storage complexity $\mathcal{O}(N_T N_S)$. 

Due to the oscillating behavior of the Helmholtz kernel, existing standard fast methods for the reduction of the complexity do not perform well for relatively large wave numbers $\kappa$. Therefore, a variety of methods have been developed. There are several versions of the  fast multipole method (FMM) based on different expansions of the Helmholtz kernel $f$. A first version suitable for high frequency regimes is given in \cite{ow:rokhlin2} and an overview of the early developments can be found in \cite{ow:nishimura}. Of further interest are the methods in \cite{ow:darve,ow:hu}, which rely on plane wave expansions, and the wideband method in \cite{ow:cheng} which switches between different expansions in low and high frequency regimes.

Directional methods allow to overcome the deficiencies of standard schemes in high frequency regimes, too. The basic idea of these methods is that the Helmholtz kernel $f$ can locally be smoothed by a plane wave. In the context of fast methods this idea was first considered in \cite{ow:brandt} and later in \cite{ow:engquist}. In \cite{ow:messnerschanz} the idea is picked up and combined with an approximation of the kernel via interpolation. \cite{ow:boerm2,ow:boerm1} follow a similar path in the context of $\mathcal{H}^2$-matrices providing a rigorous analysis. A slightly different method is proposed in \cite{ow:bebendorf3}, where the directional smoothing is combined with a nested cross approximation of the kernel.

In this paper we present a directional method in the spirit of \cite{ow:messnerschanz} based on a uniform clustering of the point sets. We choose this approach due to the applicability of the involved interpolation to other kernels and a smooth transition between low and high frequency regimes in contrast to the wideband FMM in~\cite{ow:cheng}. We give a description of the method in Sect.~\ref{ow:sec_derivation} and an asymptotic complexity analysis in Sect.~\ref{ow:sec_complexity_analysis}. While \cite{ow:messnerschanz} provides already a brief analysis we present a detailed one not unlike the one in \cite{ow:boerm2}, but focusing on points distributed in 3D volumes instead of points on 2D manifolds and allowing two distinct sets of points. In addition, we exploit the uniformity for a significant storage reduction compared to non-uniform approaches. This reduction and the claimed almost linear asymptotic behavior can be observed in our numerical tests in Sect.~\ref{ow:sec:num:examples}.

\section{Derivation of the Fast Directional Method} \label{ow:sec_derivation}
In this section we present a method for fast matrix-vector multiplications for the matrix $A$ in \eqref{ow:eq_helmholtz_matrix} based on a hierarchical partitioning of the sets of points into boxes and a directional multi-level approximation of the Helmholtz kernel~$f$ on suitable pairs of such boxes.

\subsection{Box Cluster Trees} \label{ow:sec_bct}
The desired matrix partition can efficiently be constructed from a hierarchical tree clustering of the point sets into axis-parallel boxes. In what follows we define uniform box cluster trees which are constructed by a uniform subdivision of an initial box, see, e.g., \cite{ow:greengardrokhlin}.
In particular, we construct a uniform box cluster tree~$\mathcal{T}_T$ for a given set of points $P_T = \{x_j\}_{j=1}^{N_T}$ in an axis-parallel box ${T = (a_1,b_1] \times \ldots \times (a_3,b_3] \subset \mathbb{R}^3}$ by Algorithm~\ref{ow:alg_box_ct}. As additional parameter we have the maximal number of points per leaf~$n_{\max}$.
\begin{algorithm}
\caption{Construction of a uniform box cluster tree~$\mathcal{T}_T$ \label{ow:alg_box_ct}}
\begin{algorithmic}[1] 
\State \textbf{input}: Points $P_T = \{x_j\}_{j=1}^{N_T}$ inside a box $T=(a_1,b_1]\times \ldots \times (a_3,b_3]$, maximal number $n_{\max}$ of points per leaf.
\State Construct an empty tree $\mathcal{T}_T$ and add $T$ as its root.
\State Call \Call{RefineCluster}{$T$, $\mathcal{T}_T$}
\Statex
\Function{RefineCluster}{$T=(a_1,b_1]\times \ldots \times (a_3,b_3]$, $\mathcal{T}$}
      \If{$\#\{x_j: x_j \in T\} > n_{\max}$}
      \State Compute center $c_1 = (a_1+b_1)/2$, $c_2 = (a_2+b_2)/2$, $c_3 = (a_3+b_3)/2$.
      \State Uniformly subdivide $T$  into 8 boxes $T_1 = (a_1,c_1]\times \ldots \times (a_3,c_3]$, \ldots, 
      \State $T_8 = (c_1,b_1]\times \ldots \times (c_3,b_3]$.
  		\For{$k=1$, \ldots, $8$}
  			\If{$\#\{x_j \colon x_j \in T_k\} \geq 1$}
  				\State Add $T_k$ to $\mathcal{T}$ as child of $T$.
  				\State Call \Call{RefineCluster}{$T_k$, $\mathcal{T}$}.
  			\EndIf
  		\EndFor
  \EndIf
\EndFunction
\end{algorithmic}
\end{algorithm}
We use standard notions of levels and leaves in trees known from graph theory. In addition we define
\begin{itemize}
\item the \emph{index set} $\hat{t} := \{j \in \{1, \ldots, N_T\} \colon x_j \in t\}$ for a box $t \in \mathcal{T}_T$,
\item the \emph{level sets of the tree} by $\mathcal{T}_T^{(\ell)} \new{:=} \{t \in \mathcal{T}_T \colon \level(t) = \ell\}$,
\item the \emph{depth} $p(\mathcal{T}_T) \new{:=} \max\{\level(t): t \in \mathcal{T}_T\}$ of the cluster tree $\mathcal{T}_T$,
\item the \emph{set} $\mathcal{L}_T$ \emph{of all leaves} of $\mathcal{T}_T$.
\end{itemize}
In general, Alg.~\ref{ow:alg_box_ct} creates an adaptive, i.e.~unbalanced cluster tree depending on the point distribution. Other construction principles for box cluster trees such as bisection \cite[Sect.~3.1.1]{ow:rjasanow} tailor the tree to the point sets yielding more balanced trees. However, the boxes at a given level $\ell$ of such a tree can vary strongly in shape, while the ones of a uniform box cluster tree are identical up to translation. We will exploit this uniformity to avoid recomputations and to reduce the storage costs of the presented method.

\subsection{A Directional Kernel Approximation} \label{ow:sec_dir_approx}

In this section we describe a method to approximate the Helmholtz kernel $f$ on a suitable pair of boxes~$t$ and~$s$ by a separable expansion, which will allow for low rank approximations of suitable subblocks of the matrix~$A$ in \eqref{ow:eq_helmholtz_matrix}. Due to the oscillatory part $\exp(i \kappa |x-y|)$ of $f$, standard approaches like tensor interpolation of the kernel are not effective for relatively large $\kappa$ as pointed out in \cite{ow:bebendorf3,ow:messnerschanz}. Therefore, we consider a directional approach which first appeared in \cite{ow:brandt} and~\cite{ow:engquist} and was later used in \cite{ow:boerm1} and~\cite{ow:messnerschanz} among others. The basic idea is that the oscillatory part $\exp(i\kappa |x-y|)$ of~$f$ can be smoothened by a plane wave term $\exp(-i\kappa \langle x-y, c\rangle)$ in a cone around a direction $c \in \mathbb{R}^3$ with $|c|=1$. We can rewrite the Helmholtz kernel~$f$ by expanding the numerator and the denominator by a plane wave term yielding
\begin{align} 
	f(x,y) &= f_c(x,y) \exp(i\kappa \langle x, c \rangle) \exp(-i\kappa \langle y, c \rangle), \label{ow:eq_dir_rep_helmholtza} \\ \label{ow:eq_dir_rep_helmholtz}
	f_c(x,y) &:= f(x,y) \exp(-i\kappa \langle x-y, c \rangle) = 
	\frac{\exp(i\kappa (|x-y| - \langle x-y,c \rangle))}{4\pi |x-y|}.
\end{align}
The modified kernel function $f_c$ is somewhat smoother than $f$ on suitable boxes~$t$ and~$s$. In fact, if two points $x \in t$ and $y \in s$ satisfy $(x-y)/|x-y| \approx c$, then $f_c(x,y) \approx (4 \pi |x-y|)^{-1}$, i.e.~the oscillations of $f$ are locally damped in $f_c$.
Therefore tensor interpolation can be applied to approximate $f_c$ instead of $f$ on suitable axis-parallel boxes $t$ and $s$ and we get
\begin{equation} \label{ow:eq_approx_f_c}
	f_c(x,y) \approx \sum\limits_{\nu \in M}\sum\limits_{\mu \in M} 
	  f_c(\xi_{t, \nu}, \xi_{s, \mu}) L^{(m)}_{t, \nu}(x) L^{(m)}_{s, \mu}(y),
\end{equation}
where $\nu$ and $\mu$ are multi-indices in the set $M = \{1, \ldots, m+1\}^3$, $\xi_{t, \nu}$ are tensor products of 1D Chebyshev nodes of order $m+1$ transformed to the box ${t = (a_1, b_1] \times \ldots \times (a_3, b_3]}$, i.e. $\xi_{t, \nu} = (\xi_{[a_1, b_1], \nu_1}, \xi_{[a_2, b_2], \nu_2}, \xi_{[a_3, b_3], \nu_3})$ with
\begin{equation*}
	\xi_{[a_j,b_j], \nu_j} = \frac{a_j + b_j}{2} + \frac{b_j-a_j}{2} 
	\cos \left( \frac{2 \nu_j - 1}{2 \pi (m+1)}\right), \quad \nu_j \in \{1, \ldots, m+1\},
\end{equation*}
and $L_{t,\nu}^{(m)}$ are the corresponding Lagrange polynomials, which are tensor products of the 1D Lagrange polynomials corresponding to the interpolation nodes $\{\xi_{[a_j,b_j], \nu_j}\}_{\nu_j=1}^{m+1}$.

Inserting approximation \eqref{ow:eq_approx_f_c} into \eqref{ow:eq_dir_rep_helmholtza} and grouping the terms depending on $x$ and $y$, respectively, yields the desired separable approximation
\begin{align}
	f(x,y) &\approx \sum\limits_{\nu \in M}\sum\limits_{\mu \in M} 
	  f_c(\xi_{t, \nu}, \xi_{s, \mu}) L^{(m)}_{t, c, \nu} (x) \overline{L^{(m)}_{s, c, \mu} (y)}, 
	  \label{ow:eq_dir_approx_sl} \\
	  L^{(m)}_{t, c, \nu} (x) &:= L^{(m)}_{t, \nu} (x) \exp(i\kappa \langle x, c \rangle).\label{ow:eq_lmt}
\end{align} 

The directional approximation \eqref{ow:eq_dir_approx_sl} of $f$ can be used to approximate the submatrix $A\big|_{\hat{t} \times \hat{s}}$ of the matrix $A$ in \eqref{ow:eq_helmholtz_matrix} restricted to the entries of the index sets $\hat{t}$ and $\hat{s}$ for two suitable axis-parallel boxes $t$ and $s$, i.e.,
\begin{equation} \label{ow:eq_helmholtz_matrix_entry_approx}
	A\big|_{\hat{t}\times \hat{s}}[j,k] = f(x_j, y_k) 
	\approx  \sum\limits_{\nu \in M}\sum\limits_{\mu \in M} 
	f_c(\xi_{t, \nu}, \xi_{s, \mu}) L^{(m)}_{t,c, \nu}(x_j) \overline{L^{(m)}_{s,c, \mu}}(y_k).
\end{equation}
In matrix notation this reads
\begin{equation} \label{ow:eq_helmholtz_matrix_approx_sl}
	A\big|_{\hat{t}\times \hat{s}} \approx L_{t,c} A_{c,t\times s} L_{s,c}^*,
\end{equation} 
where we define the \emph{coupling matrix} $A_{c,t\times s}\in \mathbb{C}^{(m+1)^3 \times (m+1)^3}$ by
\begin{equation} \label{ow:eq_def_coupling_matrix}
  A_{c,t\times s}[j,k] := f_c(\xi_{t,\alpha_j},\xi_{s,\beta_k}), \quad j,k \in \{1,\ldots,(m+1)^3\},
\end{equation}
for \new{suitably ordered} multi-indices $\alpha_j, \beta_k \in M= \{1, \ldots, m+1\}^3$,      
the \emph{directional interpolation matrix} $L_{t,c} \in \new{\mathbb{C}^{\hat{t} \times (m+1)^3}}$ by
\begin{equation} \label{ow:eq_dir_interpol_matrices}
	L_{t,c}[j,k] := L^{(m)}_{t, c, \alpha_k}(x_j), \quad j \in \hat{t}, k \in \{1,\ldots,(m+1)^3\}, \\
\end{equation}
and $L_{s,c}$ analogously. In particular, instead of the original $\#\hat{t} \cdot \# \hat{s}$ matrix entries only $(m+1)^3(\#\hat{t}+\#\hat{s}+(m+1)^3)$ entries have to be computed for the approximation in \eqref{ow:eq_helmholtz_matrix_approx_sl}, which is significantly less if $(m+1)^3 \ll \#\hat{t}, \#\hat{s}$.

In the following admissibility conditions we will specify for which boxes $t$ and~$s$ and which direction $c$ the approximation in \eqref{ow:eq_dir_approx_sl} is applicable. Similar criteria have been considered in \cite{ow:bebendorf3,ow:boerm1,ow:messnerschanz}. In particular, the criteria lead to exponential convergence of the approximation with respect to the interpolation degree \cite{ow:boerm1,ow:watschinger}.
\begin{definition}[{Directional admissibility \cite[cf.~Sect.~3.3]{ow:boerm1}}] \label{ow:def_dir_approx_criteria}
Let ${t,s \subset \mathbb{R}^3}$ be two axis-parallel boxes and let $c \in \mathbb{R}^3$ be a direction with $|c|=1$ or $c = 0$. Denote the midpoints of $t$ and $s$ by $m_t$ and $m_s$, respectively. Let two constants $\eta_1 > 0$ and~$\eta_2>0$ be chosen suitably. Define the diameter $\diam{t}$ and the distance $\dist{t}{s}$ by
\begin{align*}
	\diam{t} &\new{:=} \sup_{x_1,x_2 \in t} |x_1-x_2|, \quad
	\dist{t}{s} \new{:=} \inf_{x \in t, y \in s} |x-y|.
\end{align*}
We say that $t$ and $s$ are directionally admissible with respect to $c$ if the separation criterion
\begin{equation} \label{ow:eq_adm_well_sep}
	\tag{A1}
	\max\{\diam{t},\diam{s}\} \leq \eta_2 \dist{t}{s},
\end{equation}
and the two cone admissibility criteria
\begin{align}
	\tag{A2}
	\kappa \left|\frac{m_t-m_s}{|m_t-m_s|} - c \right| &\leq \frac{\eta_1}{\max\{\diam{t},\diam{s}\}}, \label{ow:eq_adm_direction}\\[8pt]
	\tag{A3}
	\kappa \max\{\diam{t},\diam{s}\}^2 &\leq \eta_2 \dist{t}{s} \label{ow:eq_adm_cone_dist}
	\vspace{-12pt}
\end{align}
are satisfied.
\end{definition}

Criterion \eqref{ow:eq_adm_well_sep} is a standard separation criterion, see, e.g., \cite{ow:greengardrokhlin} and \cite[Sect.~4.2.3]{ow:hackbuscheng}. It ensures that the boxes $t$ and $s$ are well-separated allowing for an approximation of general non-oscillating kernels.

Criterion~\eqref{ow:eq_adm_cone_dist} is similar to \eqref{ow:eq_adm_well_sep}, since it also controls the distance of two boxes $t$ and $s$. Note that \eqref{ow:eq_adm_well_sep} follows immediately from \eqref{ow:eq_adm_cone_dist} in case that $\kappa \max\{\diam{t}, \diam{s} \} > 1$ and vice versa in the opposite case. As stated in \cite[Sect.~3]{ow:boerm2}, \eqref{ow:eq_adm_cone_dist} can also be understood as a bound on the angle between all vectors $x-y$ for $x \in t$ and $y \in s$ that shrinks if $\kappa$ or $\max\{\diam{t}, \diam{s} \}$ increases. Hence, \eqref{ow:eq_adm_cone_dist} guarantees that the angle between $x-y$ and a direction~$c$ is small if the angle between the difference of the midpoints $m_t - m_s$ and $c$ is already small, which is enforced by~\eqref{ow:eq_adm_direction}.

Indeed, criterion \eqref{ow:eq_adm_direction} is used to assign a suitable direction $c$ to two non-overlapping boxes $t$ and $s$. While the choice $c = (m_t - m_s)/|m_t - m_s|$ would always guarantee \eqref{ow:eq_adm_direction}, we want to choose $c$ from a small, finite set of directions. This allows to use the same direction $c$ for a fixed box $t$ and several boxes $s_j$ and, therefore, to use the same interpolation matrix~$L_{t,c}$ for the approximation of various blocks $A\big|_{\hat{t}\times\hat{s}_j}$ as in \eqref{ow:eq_helmholtz_matrix_approx_sl}. A possible way to construct suitable sets of directions and further details on criterion \eqref{ow:eq_adm_direction} are discussed in Sect.~\ref{ow:sec_dir}. First, we want to discuss how to use criteria \eqref{ow:eq_adm_well_sep} and \eqref{ow:eq_adm_cone_dist} to construct a suitable partition of the matrix $A$ in \eqref{ow:eq_helmholtz_matrix} based on the clustering described in Sect.~\ref{ow:sec_bct}.

\subsection{Partitioning of the Matrix}

In general, the sets of evaluation points $P_T$ and $P_S$ for the matrix $A$ in \eqref{ow:eq_helmholtz_matrix} are contained in overlapping boxes~$T$ and~$S$. Therefore, the full matrix $A$ cannot be approximated directly. For this reason, we recursively construct a partition of $A$ by Alg.~\ref{ow:alg_block_ct}, which we organize in a block tree~$\mathcal{T}_{T \times S}$ (\cite[Sect.~5.5]{ow:hackbuscheng}).
\begin{algorithm}[ht]
\caption{Construction of a block tree $\mathcal{T}_{T \times S}$ \label{ow:alg_block_ct}}
\begin{algorithmic}[1] 
\State \textbf{input}: Box cluster trees $\mathcal{T}_T$ and $\mathcal{T}_S$, parameter $\eta_2$ for the criteria \eqref{ow:eq_adm_well_sep} and \eqref{ow:eq_adm_cone_dist}.
\State Set $b=(t_1^0, s_1^0)$, i.e. the pair of roots of $\mathcal{T}_T$ and $\mathcal{T}_S$.
\State Construct an empty tree $\mathcal{T}_{T \times S}$ and add $b$ as its root.
\State Call \Call{RefineBlock}{$b$, $\mathcal{T}_{T \times S}$}. 
\Statex
\Function{RefineBlock}{$b=(t,s)$, $\mathcal{T}_{T\times S}$}
	\If{$t \in \mathcal{L}_T$ \textbf{or} $s \in \mathcal{L}_S$}
		\State \Return
	\EndIf
  \If{$t$ and $s$ violate \eqref{ow:eq_adm_well_sep} or \eqref{ow:eq_adm_cone_dist}}
  		\For{$t' \in \child(t)$}
  			\For{$s' \in \child(s)$}
  				\State Add $b'=(t',s')$ to $\mathcal{T}_{T \times S}$ as child of $b$.
  				\State Call \new{\Call{RefineBlock}{$b'$, $\mathcal{T}_{T \times S}$}}.
  			\EndFor
  		\EndFor
  \EndIf
\EndFunction
\end{algorithmic}
\end{algorithm}
\begin{definition}
  \label{ow:def_block_tree}
  Let $\mathcal{T}_T$ and $\mathcal{T}_S$ be two uniform box cluster trees and let $\eta_2 > 0$. A block tree $\mathcal{T}_{T \times S}$ is constructed by Alg.~\ref{ow:alg_block_ct}.
The set of all leaves of $\mathcal{T}_{T \times S}$ is denoted by $\mathcal{L}_{T \times S}$ and split into the set of admissible (i.e.~approximable) leaves and the set of inadmissible leaves
 \begin{align*}
 	\mathcal{L}_{T \times S}^+ &:= \{b=(t,s) \in \mathcal{L}_{T \times S} : t \text{ and } s  \text{ satisfy \eqref{ow:eq_adm_well_sep} and \eqref{ow:eq_adm_cone_dist}} \},\\
\mathcal{L}_{T \times S}^- &:= \mathcal{L}_{T \times S} \ \backslash \ \mathcal{L}_{T \times S}^+.
 \end{align*} 
\end{definition}
For a given block tree $\mathcal{T}_{T \times S}$ the pairs of indices $\hat{t} \times \hat{s}$ of all leaves ${(t,s) \in \mathcal{L}_{T \times S}}$ form a partition of the full index set $\{1, \ldots, N_T\} \times \{1, \ldots, N_S\}$, i.e.~of the matrix~$A$. The matrix blocks corresponding to admissible blocks $b \in \mathcal{L}_{T \times S}^+$ can be approximated by the directional interpolation~\eqref{ow:eq_helmholtz_matrix_approx_sl}. Inadmissible blocks related to $b \in \mathcal{L}_{T \times S}^-$ are computed directly. 

\subsection{Choice of Directions} \label{ow:sec_dir}

As we would like to use relatively small numbers of directions $c$ in the directional approximations~\eqref{ow:eq_helmholtz_matrix_approx_sl}, we consider a fixed set of directions $D^{(\ell)}$ for all blocks $(t,s)$ at a given level $\ell$ of the block tree. These sets $D^{(\ell)}$ should be constructed in such a way that for all blocks $(t,s)$ at level $\ell$ in $\mathcal{L}_{T \times S}^+$ there exists a direction~${c \in D^{(\ell)}}$ such that criterion \eqref{ow:eq_adm_direction} holds for some fixed $\eta_1$.

Since the bound on the right-hand side of \eqref{ow:eq_adm_direction} increases for decreasing diameters of $t$ and $s$ and these diameters are halved for each new level of the uniform box cluster trees, the number of directions in $D^{(\ell)}$ can be reduced with increasing level~$\ell$. If the maximum of the diameters of two boxes~$t$ and~$s$ at level~$\tilde{\ell}$ is so small that the bound on the right-hand side of \eqref{ow:eq_adm_direction} is greater than~$\kappa$, then \eqref{ow:eq_adm_direction} holds for $c=0$ for all following levels. In this case, a plane wave term is not needed for the approximation of the Helmholtz kernel~$f$, and the approximation~\eqref{ow:eq_helmholtz_matrix_approx_sl} coincides with a standard tensor interpolation.
We call the \new{other} levels satisfying 
\begin{equation}
  \frac{\eta_1}{\kappa \max\{\diam{t},\diam{s}\}} \leq 1, \quad
  \text{ for all } t \in \mathcal{T}^{\ell}_T, s \in \mathcal{T}^{\ell}_S,
  \label{ow:eq_high_freq_lev}
\end{equation}
high frequency levels and denote the \emph{largest high frequency level} as $\ell_\mathrm{hf}$, or set $\ell_\mathrm{hf} = -1$ in case that all levels $\ell \geq 0$ are low frequency levels, i.e.~do not satisfy~\eqref{ow:eq_high_freq_lev}.
The value of~$\ell_\mathrm{hf}$ depends on $\eta_1$ and the uniform box cluster trees $\mathcal{T}_T$ and $\mathcal{T}_S$. In practice, we choose a suitable level $\ell_\mathrm{hf}$ instead of $\eta_1$ and construct the \emph{sets of directions} $D^{(\ell)}$, using more and more directions for levels $\ell < \ell_\mathrm{hf}$. Our construction by Alg.~\ref{ow:alg_dir} combines ideas from \cite[Sect.~4.1]{ow:engquist} and \cite[Sect.~3]{ow:boerm2}.
\begin{algorithm}[t]
\caption{Construction of directions $D^{(\ell)}$ \label{ow:alg_dir}}
\begin{algorithmic}[1]
\State \textbf{input}: Largest high frequency level $\ell_{\mathrm{hf}} \geq -1$.
\For{$\ell = \ell_\mathrm{hf}+1$, $\ell_\mathrm{hf}+2$,\ldots,$\min\{p(\mathcal{T}_T), p(\mathcal{T}_S)\}$}
	\State Set $D^{(\ell)} = \{0\}$.
\EndFor
\State Construct the six faces $\{E_j^{(\ell_\mathrm{hf})}\}_{j=1}^6$ of the cube $[-1,1]^3$, i.e. \newline
 \hspace*{10pt} $E_1^{(\ell_\mathrm{hf})} = \{-1\} \times [-1,1]^2$, 
$E_2^{(\ell_\mathrm{hf})} = \{1\} \times [-1,1]^2$, \ldots, $E_6^{(\ell_\mathrm{hf})} = [-1,1]^2 \times \{1\}$.
\State Set $D^{(\ell_\mathrm{hf})} = \{c_j^{(\ell_\mathrm{hf})}\}_{j=1}^6$ where $c_j^{(\ell_\mathrm{hf})}$ is the midpoint of $E_j^{(\ell_\mathrm{hf})}$, i.e. \newline
 \hspace*{10pt} $c_1^{(\ell_\mathrm{hf})} = (-1,0,0)$, 
$c_2^{(\ell_\mathrm{hf})} = (1,0,0)$, \ldots,
$c_6^{(\ell_\mathrm{hf})} = (0,0,1)$.
\For{$\ell = \ell_\mathrm{hf}-1, \ldots, 0$}
	\State Set $D^{(\ell)} = \emptyset$.
	\For{all faces $E_j^{(\ell+1)}$, $j=1, \ldots, 6\cdot 4^{\ell_\mathrm{hf}-\ell-1}$}
		\State Uniformly subdivide $E_j^{(\ell+1)}$ into 4 faces $E_{4(j-1)+1}^{(\ell)}$, \ldots,
		 $E_{4j}^{(\ell)}$.
		\State Construct normalized midpoints $c_{4(j-1)+1}^{(\ell)}$, \ldots, $c_{4j}^{(\ell)}$ of
		$E_{4(j-1)+1}^{(\ell)}$, \ldots, $E_{4j}^{(\ell)}$.
		\State Add directions $c_{4(j-1)+1}^{(\ell)}$, \ldots, $c_{4j}^{(\ell)}$ to $D^{(\ell)}$.
	\EndFor
\EndFor
\end{algorithmic}
\end{algorithm}

Finally, we assign a direction $c \in D^{(\ell)}$ to a pair of boxes $t$ and $s$ which is close to the normalized difference $(m_t-m_s)/|m_t-m_s|$ of the midpoints of $t$ and $s$
and, hence, can be used for the directional approximation~\eqref{ow:eq_helmholtz_matrix_approx_sl}. For this purpose, we define a mapping $\dir_{(\ell)}$ for each level $\ell \in \mathbb{N}_0$, which maps a vector $v$ in~$\mathbb{R}^3 \backslash \{0\}$ to a direction $c_j^{(\ell)}$ such that the intersection point of the ray $\{ \lambda v : \lambda > 0\}$ and the surface of the cube~$[-1,1]^3$ lies in the face $E_j^{(\ell)}$ (cf.~Alg.~\ref{ow:alg_dir}). 
\begin{definition} \label{ow:def_dir_maps}
Let $\ell_{\mathrm{hf}} \geq -1$ and let the directions $D^{(\ell)}$ and the faces $\{E_j^{(\ell)}\}$ be constructed by Alg.~\ref{ow:alg_dir}.
We define the mapping $\dir_{(\ell)}: \mathbb{R}^3 \rightarrow D^{(\ell)} \cup \{0\}$ for each $\ell \in \mathbb{N}_0$ as follows:
\begin{itemize}
\item If $\ell > \ell_{\mathrm{hf}}$ we set $\dir_{(\ell)}(v)=0$ for all $v \in \mathbb{R}^3$.
\item If $\ell \leq \ell_{\mathrm{hf}}$ we set $\dir_{(\ell)}(0)=0$. For all $v \in \mathbb{R}^3 \backslash \{0\}$ we set  $\dir_{(\ell)}(v) = c_{j(v)}^{(\ell)}$ where
\begin{align*}
  j(v) := \min\{j: \psi_Q(v) \in E_j^{(\ell)}\},
  \quad \psi_Q(v) := \frac{1}{\max\limits_{j \in \{1,\ldots, 3\}} |v_j|} \ v
\end{align*}
to avoid ambiguity.
\end{itemize}
For two boxes $t, s \subset \mathbb{R}^3$ and a level $\ell \in \mathbb{N}_0$ we define the direction $c_{(\ell)}(t,s)$ by
\begin{equation*}
	c_{(\ell)}(t,s) := \dir_{(\ell)}\left(\frac{m_t - m_s}{|m_t - m_s|}\right).
\end{equation*}
\end{definition}
In this way \eqref{ow:eq_adm_direction} is satisfied for two boxes $t$, $s$, and the direction $c_{(\ell)}(t,s)$ for a constant $\eta_1$ which depends linearly on the product $\kappa q_{\ell_\mathrm{hf}}$ \cite[Thm.~2.19]{ow:watschinger}. Here $q_{\ell_\mathrm{hf}}$ denotes the maximal diameter of all boxes at level $\ell_\mathrm{hf}$ in the trees $\mathcal{T}_T$ and~$\mathcal{T}_S$.

\subsection{Transfer Operations}

The approximation of an admissible subblock $A\big|_{\hat{t}\times \hat{s}}$ of $A$ in \eqref{ow:eq_helmholtz_matrix_entry_approx} can be further enhanced. If $t$ is a non-leaf box at level $\ell$ in a box cluster tree $\mathcal{T}_T$ with children $t_1,\ldots, t_k$, the directional interpolation matrix $L_{t,c}$ can be approximated using the matrices $L_{t_j,c_{\ell+1}}$ for a suitable direction $c_{\ell+1}$. We describe this approach following \cite[Sect.~2.2.2]{ow:boerm1}.

Let us rewrite the generating functions $L^{(m)}_{t, c, \nu}(x)$ of $L_{t,c}$  in~\eqref{ow:eq_lmt} by
\begin{equation*}
	L^{(m)}_{t, c, \nu}(x) = \exp(i \kappa \langle x, c_{\ell+1} \rangle ) 
	\left[ \exp(i \kappa \langle x, c - c_{\ell+1} \rangle ) L^{(m)}_{t, \nu}(x) \right].
\end{equation*}
If $c_{\ell+1}$ is sufficiently close to $c$, the term in square brackets is smooth and can be interpolated for points $x$ in a child box $t_j$ yielding
\begin{equation*}
	\exp(i \kappa \langle x, c - c_{\ell+1} \rangle ) L^{(m)}_{t, \nu}(x) \approx 
	\sum\limits_{\tilde{\nu} \in M} 
	\exp( i \kappa \langle \xi_{t_{j},\tilde{\nu}}, c - c_{\ell+1} \rangle)
	L^{(m)}_{t, \nu}(\xi_{t_{j},\tilde{\nu}}) L^{(m)}_{t_{j}, \tilde{\nu}} (x).
\end{equation*}
This provides an approximation of the restriction of $L^{(m)}_{t, c, \nu}$ to the child $t_j$
\begin{equation*}
	L^{(m)}_{t, c, \nu}\big|_{t_j}(x) \approx \sum\limits_{\tilde{\nu} \in M} 
	\left( \left[ \exp( i \kappa \langle \xi_{t_{j},\tilde{\nu}}, c - c_{\ell+1} \rangle)
	L^{(m)}_{t, \nu}(\xi_{t_{j},\tilde{\nu}}) \right] L^{(m)}_{t_{j}, c_{\ell+1}, \tilde{\nu}} (x) \right).
\end{equation*}
In matrix notation the related restriction to the index set $\hat{t}_j$ reads as
\begin{equation} \label{ow:eq_approx_dir_interpol_mat}
	L_{t, c}|_{\hat{t}_j \times (m+1)^3} \approx L_{t_j, c_{\ell+1}} E_{t_j, c},
\end{equation}
where the entries of the transfer matrix $E_{t_{j}, c} \in \mathbb{C}^{(m+1)^3 \times (m+1)^3}$ are defined by
\begin{equation} \label{ow:eq_def_transfer_matrix}
	E_{t_{j}, c}[k, \ell] :=  
	\exp( i \kappa \langle \xi_{t_{j},\nu_k}, c - c_{\ell+1} \rangle)
	L^{(m)}_{t, \nu_\ell}(\xi_{t_{j},\nu_k}),
\end{equation}
for all $k, \ell \in \{1, \ldots, (m+1)^3\}$.

A suitable choice \cite[Thm.~2.19]{ow:watschinger} for the direction $c_{\ell+1}$ is given by $\dir_{(\ell+1)}(c)$, with $\dir_{(\ell+1)}$ given in Def.~\ref{ow:def_dir_maps}. Since this direction depends only on $c$ and the level~$\ell$ of the box $t$, it is reasonable to omit the dependence of the transfer matrix~$E_{t_{j}, c}$ on $c_{\ell+1}$ in the notation.

\subsection{Main Algorithm} \label{ow:sec_main_algorithm}

In the previous sections, we have described how to partition the matrix~\eqref{ow:eq_helmholtz_matrix} and how to approximate suitable subblocks. Here we explain the complete algorithm for a a matrix-vector multiplication $g = A v$.

The idea is to execute the multiplication blockwise according to the partition induced by the leaves~$\mathcal{L}_{T \times S}$  of the block tree $\mathcal{T}_{T \times S}$. Inadmissible blocks from $\mathcal{L}_{T \times S}^-$ are multiplied directly with the target vector~$v$. For admissible blocks from $\mathcal{L}_{T \times S}^+$  we use the decomposition \eqref{ow:eq_helmholtz_matrix_approx_sl} and split the multiplication into three phases. This is similar to the usual three-phase algorithm for $\mathcal{H}^2$-matrices~\cite[Sect.~8.7]{ow:hackbuscheng} and the FMM~\cite{ow:greengardrokhlin} with adaptations due to the directional approximation. We describe the scheme first for one block corresponding to an admissible pair of boxes $t$ and $s$ at level $\ell$ and then give a description of the complete algorithm.

In the first phase, the forward transformation, the product $\tilde{v}_{s,c} := L_{s,c}^* v|_{\hat{s}}$ is computed. If $s$ is a leaf in the cluster tree, this is done directly by~\eqref{ow:eq_dir_interpol_matrices}. This is also known as S2M (source to moment) step in fast multipole methods. If $s$ is not a leaf, approximation \eqref{ow:eq_approx_dir_interpol_mat} with $c_{\ell+1} = \dir_{(\ell+1)}(c)$ is used \new{iteratively} to get 
\begin{equation*} 
	\new{\tilde{v}_{s,c} := 
	\sum_{s_j \in \child(s)} E_{s_j,c}^* \tilde{v}_{s_j,c_{\ell+1}} 
		\approx \sum_{s_j \in \child(s)} E_{s_j,c}^* 
			\left[ L_{s_j,c_{\ell+1}}^* v|_{\hat{s}_j} \right]
		\approx L_{s,c}^* v|_{\hat{s}},}
\end{equation*}
by using the products of the children, which is also known as M2M operation (moment to moment).
In the second phase, which is called multiplication phase or M2L (moment to local) step, the product 
\begin{equation*}
	\tilde{g}_{t,c} := A_{c, t \times s} \tilde{v}_{s,c}
\end{equation*}
is computed by \eqref{ow:eq_def_coupling_matrix}. In the complete algorithm all contributions from various boxes $s$ are added up, i.e.
\begin{equation*}
	\tilde{g}_{t,c} := \sum_{s: (t,s) \in \mathcal{L}_{T \times S}^+} A_{c, t \times s} \tilde{v}_{s,c}.
\end{equation*}
In the third phase, the so-called backward transformation, the product 
\begin{equation} \label{ow:eq_l2t}
	\tilde{g}|_{\hat{t}} := L_{t,c} \tilde{g}_{t,c}
\end{equation}
is computed. If $t$ is a leaf, this is done directly.  This step is known as L2T (local to target) in fast multipole methods. If $t$ is not a leaf, the approximation~\eqref{ow:eq_approx_dir_interpol_mat} is used to compute 
\begin{equation} \label{ow:eq_l2l}
	\tilde{g}_{t_j,c_{\ell + 1}} = E_{t_j,c}\ \tilde{g}_{t,c},
\end{equation}
for all children $t_j$ of $t$, which is also known as L2L operation (local to local), and the evaluation \eqref{ow:eq_l2t} takes place for descendants which are leaves. In the complete algorithm the local contribution in \eqref{ow:eq_l2l} is added to the existing contribution $\tilde{g}_{t_j, c_{\ell+1}}$ originating from the multiplication phase.

Before we present the complete Alg.~\ref{ow:alg_fdmvm}, we define the sets of active and inherited directions for each box in the cluster trees $\mathcal{T}_T$ and $\mathcal{T}_S$. These are used to keep track of all required directions for boxes $t$ and $s$ in the trees $\mathcal{T}_T$ and $\mathcal{T}_S$. They can be generated during the construction of the block tree $\mathcal{T}_{T \times S}$.
\begin{definition} \label{ow:def_active_directions}
Let $\mathcal{T}_T$ and $\mathcal{T}_S$ be two uniform box cluster trees, $\mathcal{T}_{T \times S}$ the corresponding block tree and $\ell_\mathrm{hf} \geq -1$. Recalling Alg.~\ref{ow:alg_dir} and Def.~\ref{ow:def_dir_maps} we define for all~$\ell \geq 0$ and all $t \in \mathcal{T}_T^{(\ell)}$ the set of active directions by
\begin{equation*}
	D(t) := \{c \in D^{(\ell)} : \exists \ s \in \mathcal{T}_S^{(\ell)}
	\text{ such that } (t,s) \in \mathcal{L}_{T \times S}^+ \text{ and } c = c_{(\ell)}(t,s) \}.
\end{equation*}
The set of inherited directions $\hat{D}(t)$ is defined recursively by setting $\hat{D}(t_1^0) = \emptyset$ for the root $t_1^0$ of $\mathcal{T}_T$, and for all~$\ell > 0$ and all $t \in \mathcal{T}_T^{(\ell)}$ by setting
\begin{equation*}
	\hat{D}(t) := \{\hat{c} \in D^{(\ell)} : 
	\exists \ c \in D(t') \cup \hat{D}(t') \text{ such that } \hat{c} = \dir_{(\ell)}(c),\ t' = \parent(t) \}.
\end{equation*}
Analogously, the sets of active directions $D(s)$ and inherited directions $\hat{D}(s)$ are defined for clusters $s \in \mathcal{T}_S$.
\end{definition}
\begin{algorithm}[H]
\caption{Fast directional matrix vector multiplication $g \approx A v$ \label{ow:alg_fdmvm}}
\begin{algorithmic}[1] 
\State \textbf{input}: Box cluster trees $\mathcal{T}_T$ and $\mathcal{T}_S$, block tree $\mathcal{T}_{T \times S}$, interpolation degree $m$,

sets of directions $D(t)$, $\hat{D}(t)$, $D(s)$, $\hat{D}(s)$ for all boxes $t$, $s$ in $\mathcal{T}_T$, $\mathcal{T_S}$.
\State Initialize $g = 0$.
\State \Comment{Forward transformation}
\For{all leaves $s \in \mathcal{L}_S$} \label{ow:alg_line_start_s2m}
	\For{all directions $c \in D(s) \cup \hat{D}(s)$}
		\State Compute $\tilde{v}_{s,c} = L_{s,c}^* v|_{\hat{s}}$. \label{ow:alg_line_end_s2m}
	\EndFor
\EndFor

\For{all levels $\ell = p(\mathcal{T}_S) - 1$, \ldots, $0$} \label{ow:alg_line_start_m2m}
	\For{all non-leaf boxes $s \in \mathcal{T}^{(\ell)}_S \backslash \mathcal{L}_S$} 
		\For{all directions $c \in D(s) \cup \hat{D}(s)$}
			\State Set $\tilde{v}_{s,c} = 0$.
			\For{all $s' \in \child(s)$}
				\State Update $\tilde{v}_{s,c} \mathrel{+}=
															E_{s',c}^* \tilde{v}_{s',c'}$, 
				where $c' = \dir_{(\ell+1)}(c)$. \label{ow:alg_line_end_m2m}
			\EndFor
		\EndFor
	\EndFor
\EndFor

\State \Comment{Multiplication phase}
\For{all boxes $t \in \mathcal{T}_T$} 
	\For{all directions $c \in D(t) \cup \hat{D}(t)$}
		\State Initialize $\tilde{g}_{t,c}=0$. 
	\EndFor
	\For{all boxes $s \in \mathcal{T}_S$ such that \label{ow:alg_line_m2l_loop}
			 $(t,s) \in \mathcal{L}^+_{T \times S}$}
		\State Update $\tilde{g}_{t,c} \mathrel{+}= 
									 A_{c, t \times s} \tilde{v}_{s,c}$,
		where $c = \dir_{(\ell)}(t,s)$ and $\ell = \level(t)$.
	\EndFor
\EndFor

\State \Comment{Backward transformation}
\For{all levels $\ell = 0$, \ldots, $p(\mathcal{T}_T) - 1$} \label{ow:alg_line_start_l2l}
	\For{all non-leaf boxes $t \in \mathcal{T}^{(\ell)}_T \backslash \mathcal{L}_T$} 
		\For{all directions $c \in D(t) \cup \hat{D}(t)$}
			\For{all $t' \in \child(t)$}
				\State Update $\tilde{g}_{t',c'} \mathrel{+}=
															E_{t',c} \tilde{g}_{t,c}$, 
				where $c' = \dir_{(\ell+1)}(c)$. \label{ow:alg_line_end_l2l}
			\EndFor
		\EndFor
	\EndFor
\EndFor

\For{all leaves $t \in \mathcal{L}_T$}  \label{ow:alg_line_start_l2t}
	\For{all directions $c \in D(t) \cup \hat{D}(t)$}
		\State Update $g|_{\hat{t}} \mathrel{+}= L_{t,c} \tilde{g}_{t,c}$. 
		\label{ow:alg_line_end_l2t}
	\EndFor
\EndFor

\State \Comment{Nearfield evaluation}
\For{all blocks $b=(t,s) \in \mathcal{L}^-_{T \times S}$} \label{ow:alg_line_start_nearfield}
	\State Update $g|_{\hat{t}} \mathrel{+}= 
								A|_{\hat{t} \times \hat{s}} v|_{\hat{s}}$. \label{ow:alg_line_end_nearfield}
\EndFor
\end{algorithmic}
\end{algorithm}

\subsection{Implementation Details} \label{ow:sec_implementation_details}
In this section, we describe how to exploit the uniformity of the box cluster trees \new{to reduce the storage required by the transfer matrices~$E_{t',c}$ defined in~\eqref{ow:eq_def_transfer_matrix} and the coupling matrices $A_{c,t \times s}$ defined in \eqref{ow:eq_def_coupling_matrix}. This is crucial as there is a large number of such matrices involved in the computations in Alg.~\ref{ow:alg_fdmvm}.}

For a level $\ell \geq 0$, a box $t \in \mathcal{T}_T^\ell$, a child $t'$ and directions $c$ and ${c'=\dir_{(\ell+1)}(c)}$ we consider the transfer matrix $E_{t',c}$ which has the entries
\begin{equation*}
	E_{t',c}[j, k] = 
	\exp(i \kappa \langle \xi_{t', \nu_{j}}, c - c' \rangle) L_{t,\nu_k}^{(m)}(\xi_{t', \nu_j}), 
	\quad j,k \in \{1,\ldots, (m+1)^3\}.
\end{equation*} 
This matrix can be split into a directional and a non-directional part by
\begin{equation*} 
	E_{t',c} = E^\mathrm{d}_{t',c} E_{t'},
\end{equation*}
where we define the directional part $E^\mathrm{d}_{t',c}$ and the non-directional part $E_{t'}$  by
\begin{align*}
	E^\mathrm{d}_{t',c} 
	&:= \diag\left(\{\exp(i \kappa \langle \xi_{t', \tilde{\nu}}, c - c' \rangle)\}_{\tilde{\nu} \in M}\right), \\
	E_{t'}[j, k] &:= L_{t,\nu_k}^{(m)}(\xi_{t', \nu_j}), \quad j,k \in \{1,\ldots, (m+1)^3\}.
\end{align*}
Let us consider the non-directional part $E_{t'}$ first. The value of the Lagrange polynomial $L_{t, \nu}^{(m)}$ depends only on the position of the evaluation point $\xi_{t', \mu}$ relative to the box $t$. Together with the uniformity of the box cluster tree $\mathcal{T}_T$, this implies that each $E_{t'}$ is identical to one of 8 non-directional transfer matrices in a reference configuration. Only these reference matrices of size $(m+1)^3\times (m+1)^3$ have to be computed and stored. 
The directional part $E^\mathrm{d}_{t',c}$ changes for varying boxes $t$, child boxes~$t'$ or directions $c$. Since it is diagonal, however, only $(m+1)^3$ entries instead of $(m+1)^6$ entries need to be computed. Furthermore, for low frequency levels~{$\ell > \ell_\mathrm{hf}$} the directional part $E^\mathrm{d}_{t',c}$ becomes the identity and no additional computations are required.

Next we consider the coupling matrices $A_{c,t\times s}$ defined in \eqref{ow:eq_def_coupling_matrix} for admissible blocks $(t,s)$ in a block tree $\mathcal{T}_{T \times S}$. $A_{c,t\times s}$ depends on the difference of the cluster centers only, see~\eqref{ow:eq_dir_rep_helmholtz}. Due to the uniformity of the box cluster trees, many of the coupling matrices coincide.
In particular, it suffices to compute and store all required coupling matrices for all levels $\ell$ only once for a reference configuration and assign them to the appropriate blocks $(t,s) \in \mathcal{L}_{T \times S}^+$.

The dimension of the coupling matrices \eqref{ow:eq_def_coupling_matrix} increases cubically in the interpolation degree $m$. A compression of these matrices by a low rank approximation
\begin{equation*}
	A_{c,t \times s} \approx U_{c,t \times s} V_{c,t \times s}^*,
\end{equation*}
with $U_{c,t \times s}$, $V_{c,t \times s} \in \mathbb{C}^{(m+1)^3\times k}$ for some low rank $k$, increases the performance of the algorithm (cf.~\cite{ow:messnerschanz}). Such approximations exist because the coupling matrices are generated by smooth functions. For their construction, we apply a partially pivoted ACA \cite{ow:bebendorfrjasanow,ow:rjasanow} in our implementation and the examples in Sect.~\ref{ow:sec:num:examples}, but do not analyze its effect on the complexity in the following section. A more involved compression strategy is described in~\cite{ow:boerm4}.

\section{Complexity Analysis} \label{ow:sec_complexity_analysis}
To analyze the complexity of Alg.~\ref{ow:alg_fdmvm} for fast directional matrix vector multiplications, we estimate the number of directional interpolation matrices and transfer matrices in Thm.~\ref{ow:th_nr_transfer_mat}, give then an estimate for the number of coupling matrices in Thm.~\ref{ow:thm_nr_coupling_mat} and~\ref{ow:thm_stored_coupling_mat} and finally estimate the number of nearfield matrices in Thm.~\ref{ow:thm_est_nearfield_entries}.
We start by establishing the general setting.

Throughout this section we fix the wave number $\kappa > 0$ and the sets of points $P_T=\{x_j\}_{j=1}^{N_T}$ and $P_S=\{y_k\}_{k=1}^{N_S}$, which may but do not have to coincide, and set $N = \max\{N_T, N_S\}.$ In all considerations $\mathcal{T}_T$ and $\mathcal{T}_S$ denote two uniform box cluster trees as constructed in Alg.~\ref{ow:alg_box_ct} for a fixed parameter $n_{\max}$. We set the maximum and the minimum of the depths of the trees $\mathcal{T}_T$ and $\mathcal{T}_S$
\begin{equation*}
  p_{\max} := \max\{p(\mathcal{T}_T), p(\mathcal{T}_S)\}, \quad
  p_{\min} := \min\{p(\mathcal{T}_T), p(\mathcal{T}_S)\}.
\end{equation*}
The diameters of all boxes at a fixed level $\ell$ of $\mathcal{T}_T$ are identical and denoted as $q_\ell(\mathcal{T}_T)$ just like the diameters $q_\ell(\mathcal{T}_S)$ of boxes at level $\ell$ in $\mathcal{T}_S$. For all levels $\ell \leq p_{\min}$ we define
\begin{equation*}
	q_\ell := \max\{q_\ell(\mathcal{T}_T), q_\ell(\mathcal{T}_S)\}.
\end{equation*}
The related block tree~$\mathcal{T}_{T \times S}$ is constructed by Alg.~\ref{ow:alg_block_ct} for a fixed parameter~$\eta_2$. For the directional approximation we use a small, fixed interpolation degree $m$ and the directions $D^{(\ell)}$, constructed by Alg.~\ref{ow:alg_dir} for a fixed choice of the largest high frequency level $\ell_\mathrm{hf} \geq -1$.

For the complexity analysis we will need a few assumptions which we collect and discuss here. We assume that there exist small constants $c_{\text{geo}}$, $c_{\max}$, $c_{\text{ad}}$ and $c_{\text{un}} \in \mathbb{R}_+$ such that the following assumptions hold true:
\begin{align} 
	n_{\max} &\leq c_{\max} (m+1)^3, \label{ow:eq_assume_sparse_leafs} \\
	q_0 &\leq c_{\text{un}} \min\{q_0(\mathcal{T}_T), q_0(\mathcal{T}_S)\},
	\label{ow:eq_assume_comparable_trees} \\
	p_{\max} &\leq \log_8(N) + c_{\text{ad}}, \label{ow:eq_assume_tree_uniformity} \\
	\kappa q_0 &\leq c_{\text{geo}} \sqrt[3]{N}. \label{ow:eq_assume_resolve_wavelength}
\end{align}
In addition, $\ell_\mathrm{hf}$ is assumed to be chosen such that
\begin{equation}
	\ell_\mathrm{hf} +1 \leq p_{\max} + c_{\mathrm{hf}}, \label{ow:eq_assume_reasonable_lhf}
\end{equation}
for a small constant $c_\mathrm{hf} \in \mathbb{N}_0$. Furthermore, we require that \eqref{ow:eq_assume_few_uncontrollable_leaves} holds, which we introduce and discuss~later.
Let us shortly discuss above assumptions. By equation \eqref{ow:eq_assume_sparse_leafs} we ensure that the maximal number of points in leaf boxes of the cluster trees is reasonably small. Assumption \eqref{ow:eq_assume_comparable_trees} means that the diameters of the root boxes of $\mathcal{T}_T$ and~$\mathcal{T}_S$ should be of comparable size. While this is not satisfied in general, one can enforce it by an initial subdivision of the greater box and application of the method to the resulting subboxes.
Eqn.~\eqref{ow:eq_assume_tree_uniformity} is an indirect assumption on the sets of points~$P_T$ and $P_S$, which holds if the points are distributed more or less uniformly in a 3D domain. \new{ Also Eqn.~\eqref{ow:eq_assume_resolve_wavelength} is reasonable only if points are distributed rather uniformly in a 3D volume, and guarantees that the wave length $\lambda = 2 \pi / \kappa$ is resolved in that case, which is required in typical physical applications.} Finally, Eqn.~\eqref{ow:eq_assume_reasonable_lhf} is a bound on the largest high frequency level~$\ell_\mathrm{hf}$ and allows to bound the number of directions constructed in~Alg.~\ref{ow:alg_dir}. With these assumptions we can start with the complexity analysis, which is based on the following obvious, but important observation.

\begin{remark}\label{ow:rem:one}
In Alg.~\ref{ow:alg_fdmvm} every directional interpolation matrix $L_{t,c}$ and $L_{s,c}$, every transfer matrix~$E_{t',c}$ and $E_{s',c}$, every coupling matrix $A_{c,t \times s}$ and every nearfield matrix $A|_{\hat{t} \times \hat{s}}$ is multiplied with a suitable vector exactly once. All entries of these matrices can be computed with $\mathcal{O}(1)$ operations. Since the complexity of the application of a matrix to a vector is proportional to the number of its entries, it suffices to count all these matrices and their respective entries to estimate the storage and runtime complexity of Alg.~\ref{ow:alg_fdmvm}.
\end{remark}

\begin{theorem}\label{ow:th_nr_transfer_mat}
Let assumption \eqref{ow:eq_assume_reasonable_lhf} hold true. Then there exists a constant $c_{\mathrm{LE}}$ depending only on $c_\mathrm{hf}$ such that  \new{the number $N_{\mathrm{LE}}$ of applied transfer matrices~$E_{t',c}$ and $E_{s',c}$ and directional interpolation matrices $L_{t,c}$ and~$L_{s,c}$} in Alg.~\ref{ow:alg_fdmvm} is bounded~by
\begin{equation} \label{ow:eq_estimate_transfer_and_interpol_total}
	\new{ N_{\mathrm{LE}} \leq c_{\mathrm{LE}} \, 8^{p_{\max}}. }
\end{equation}
If \eqref{ow:eq_assume_sparse_leafs} and \eqref{ow:eq_assume_tree_uniformity} apply in addition, these matrices can be stored and applied with complexity $\mathcal{O}(N)$.
\end{theorem}

\begin{proof}
\new{
We start to estimate the number $N_{\mathrm{LE},T}$ of applied transfer matrices $E_{t',c}$ for L2L operations in lines \ref{ow:alg_line_start_l2l}--\ref{ow:alg_line_end_l2l} of Alg.~\ref{ow:alg_fdmvm} and directional interpolation matrices~$L_{t,c}$ for L2T operations in lines \ref{ow:alg_line_start_l2t}--\ref{ow:alg_line_end_l2t}. For this purpose we estimate the number of such matrices for each box $t$ in $\mathcal{T}_T$.

Let us first assume, that $t \in \mathcal{T}_{T}^{\ell}$ is a non-leaf box at level $\ell$. In this case a transfer matrix is applied for each direction $c \in D(t) \cup \hat{D}(t)$ and each box $t' \in \child(t)$, but no directional interpolation matrix. The number of directions in $D(t) \cup \hat{D}(t)$ is bounded by $\# D^{(\ell)}$, which is $6\cdot 4^{\ell_\mathrm{hf}-\ell}$ if~$\ell \leq \ell_\mathrm{hf}$ and 1 else, and $\# \child(t) \leq 8$ for all $t$ due to the uniformity of the box cluster tree. Therefore, the total number $N_{\mathrm{LE}}(t)$ of transfer and directional interpolation matrices needed for a non-leaf box $t \in \mathcal{T}_{T}^{\ell}$ is bounded by 
\begin{equation*}
	B^{(\ell)} =
	\begin{cases}
		48\cdot 4^{\ell_\mathrm{hf}-\ell}, & \text{ if } \ell \leq \ell_\mathrm{hf}, \\
		8, & \text{ otherwise.}
	\end{cases}
\end{equation*}

If $t \in \mathcal{T}_{T}^{\ell}$ is a leaf box then we only need a directional interpolation matrix for each direction $c \in D(t) \cup \hat{D}(t)$ but no transfer matrix. Therefore, $N_{\mathrm{LE}}(t)$ is bounded by $6\cdot 4^{\ell_\mathrm{hf}-\ell}$ if $\ell \leq \ell_\mathrm{hf}$ and by~1 otherwise. Since this bound is less than~$B^{(\ell)}$ for all levels $\ell$, there holds $N_{\mathrm{LE}}(t) \leq B^{(\ell)}$ for all boxes $t \in \mathcal{T}_T^{\ell}$.

The number $N_{\mathrm{LE},T}$ of all directional interpolation matrices and transfer matrices for boxes $t \in \mathcal{T}_T$ can hence be estimated by
\begin{equation*}
	N_{\mathrm{LE},T} = \sum_{\ell=0}^{p(\mathcal{T}_T)} \sum_{t \in \mathcal{T}_T^{\ell}} N_{\mathrm{LE}}(t)
		\leq \sum_{\ell=0}^{p(\mathcal{T}_T)} \# \mathcal{T}_T^{\ell} \, B^{(\ell)}.
\end{equation*} 
Due to the uniformity of the box cluster tree there holds $\#\mathcal{T}_T^{\ell} \leq 8^\ell$. Let us first assume that all levels in $\mathcal{T}_T$ are high frequency levels, i.e. $p(\mathcal{T}_T) \leq \ell_\mathrm{hf}$. Then we can further estimate
\begin{equation}
	N_{\mathrm{LE},T} \leq \sum_{\ell=0}^{p(\mathcal{T}_T)} 48 \cdot 4^{\ell_\mathrm{hf}-\ell}\, 8^{\ell}
		< 48\cdot 4^{\ell_\mathrm{hf}}\,2^{p(\mathcal{T}_T)+1} \leq 24 \cdot 4^{c_\mathrm{hf}} \, 8^{p_{\max}},
\end{equation}
where we used assumption \eqref{ow:eq_assume_reasonable_lhf} in the last step. If instead $p(\mathcal{T}_T) > \ell_\mathrm{hf}$, we get
\begin{equation} 
	\begin{split}
		N_{\mathrm{LE},T} &\leq \sum_{\ell=0}^{\ell_{\mathrm{hf}}} 48 \cdot 4^{\ell_\mathrm{hf}-\ell}\, 8^{\ell} 
			+ \sum_{\ell=\ell_{\mathrm{hf}}+1}^{p(\mathcal{T}_T)} 8^{\ell+1} \\
		&\leq 12 \cdot 8^{\ell_\mathrm{hf}+1} + 8\, (8^{p(\mathcal{T}_T)+1} - 8^{\ell_\mathrm{hf}+1}) 
			\leq 68 \cdot 8^{p_{\max}}.
	\end{split}
\end{equation}

Analogously, we can estimate the number $N_{\mathrm{LE},S}$ of transfer matrices and directional interpolation matrices needed for the S2M and M2M operations in lines~\ref{ow:alg_line_start_s2m}--\ref{ow:alg_line_end_m2m} of Alg.~\ref{ow:alg_fdmvm}. Therefore, the estimate on the number $N_{\mathrm{LE}}$ of all transfer and directional interpolation matrices in \eqref{ow:eq_estimate_transfer_and_interpol_total} holds with ${c_{\mathrm{LE}} = 2 \cdot \max( 68, 24 \cdot 4^{c_\mathrm{hf}} )}$.

To prove the complexity statement we observe that every transfer matrix~\eqref{ow:eq_def_transfer_matrix} has $(m+1)^6$ entries and every directional interpolation matrix~\eqref{ow:eq_dir_interpol_matrices} has at most $n_{\max} (m+1)^3 \leq c_{\max}(m+1)^6$ entries by assumption~\eqref{ow:eq_assume_sparse_leafs}. Therefore, the linear complexity is a direct consequence of \eqref{ow:eq_estimate_transfer_and_interpol_total}, if in addition \eqref{ow:eq_assume_tree_uniformity} holds. \qed
}		
\end{proof}

\begin{theorem} \label{ow:thm_nr_coupling_mat}
Let assumption \eqref{ow:eq_assume_comparable_trees} hold true. Then there exists a constant $c_\mathrm{C}$ depending only on $c_\mathrm{un}$ and $\eta_2$, such that the number $N_{\mathrm{C}}$ of all coupling matrices $A_{c,t \times s}$ in Alg.~\ref{ow:alg_fdmvm} is bounded by
\begin{equation} \label{ow:eq_est_nr_coupling_mat}
	N_{\mathrm{C}} \leq c_\mathrm{C} \left(p_{\min} (q_0 \kappa)^3 + 8^{p_{\min}}\right).
\end{equation}
If in addition \eqref{ow:eq_assume_tree_uniformity} and \eqref{ow:eq_assume_resolve_wavelength} hold true, these matrices can be stored and applied with complexity $\mathcal{O}(N \log(N))$. If \eqref{ow:eq_assume_resolve_wavelength} is replaced by the stronger assumption
\begin{equation} \label{ow:eq_assume_overresolve_wavelength}
	\kappa q_0 \leq c\sqrt[3]{N/\log(N)},
\end{equation}
then the complexity is reduced to $\mathcal{O}(N)$.
\end{theorem}

\begin{proof} In this proof we pursue similar ideas as in~\cite[cf.~proof of Lem.~8]{ow:boerm2}.
We assume that the depth of $\mathcal{T}_{T \times S}$ is not zero, because \new{otherwise} $N_{\mathrm{C}} \leq 1$ and the assertion is trivial. Our strategy is to estimate the numbers $N_{\mathrm{C}}^{(\ell)}$ of coupling matrices at all relevant levels $\ell = 1,\ldots, p_{\min}$. 

In line \ref{ow:alg_line_m2l_loop} of Alg.~\ref{ow:alg_fdmvm} we see that the number of coupling matrices needed for a box $t \in \mathcal{T}_T^{(\ell)}$ is given by $\#\{s: (t,s) \in \mathcal{L}_{T \times S}^+ \}$. For such blocks $(t,s) \in \mathcal{L}_{T \times S}^+$ the  $\parent(s)$ is in the nearfield $\mathcal{N}(t_\mathrm{p})$ of $t_\mathrm{p} := \parent(t)$  by construction of the block tree in Alg.~\ref{ow:alg_block_ct}, where 
\begin{equation*}
 \quad \mathcal{N}(t_\mathrm{p}) := \{s_\mathrm{p} \in \mathcal{T}_S^{(\ell-1)}: s_\mathrm{p} \text{ and } t_\mathrm{p} \text{ violate \eqref{ow:eq_adm_well_sep} or \eqref{ow:eq_adm_cone_dist}} \}.
\end{equation*}
Using this property and the uniformity of the box cluster trees we can estimate
\begin{equation} \label{ow:eq_est_n_c_ell}
	N_{\mathrm{C}}^{(\ell)} 
	= \sum_{t \in \mathcal{T}_T^{(\ell)}} \sum_{(t,s) \in \mathcal{L}_{S \times T}^+} 1 
	\leq \sum_{t \in \mathcal{T}_T^{(\ell)}} 8 \cdot \# \mathcal{N}(\parent(t)) \leq 8^{\ell+1}
	N_{\mathcal{N},T}^{(\ell-1)},
\end{equation}
where $N_{\mathcal{N},T}^{(\ell-1)}$ is an upper bound for the number of boxes in the nearfield of a box at level $\ell-1$ in $\mathcal{T}_T$ which we estimate in the following.

We cover the nearfield~$\mathcal{N}(t)$ of a fixed box $t \in \mathcal{T}_T^{(\ell)}$ by a ball $B_{r_\ell}(m_t)$ with radius $r_\ell$ and center $m_t$ and take the ratio of the volume of the ball and the one of a box to estimate $N_{\mathcal{N},T}^{(\ell)}$ for $\ell \geq 1$. We have to distinguish the cases of the two admissibility criteria \eqref{ow:eq_adm_well_sep} and \eqref{ow:eq_adm_cone_dist}. For this purpose, let $\tilde{\ell}$ be such that $\kappa q_j > 1$, if and only if $j \leq \tilde{\ell}$. Such an $\tilde{\ell}$ exists since $q_j$ decreases monotonically for increasing level~$j$. In particular, we set $\tilde{\ell}=-1$, if $\kappa q_j \leq 1$ for all $j \geq 0$. If $j \leq \tilde{\ell}$ criterion~\eqref{ow:eq_adm_cone_dist} implies \eqref{ow:eq_adm_well_sep} as mentioned in Sect.~\ref{ow:sec_dir_approx}. Vice versa, \eqref{ow:eq_adm_well_sep} implies \eqref{ow:eq_adm_cone_dist} if~$j > \tilde{\ell}$.

Let us first assume that $\ell \leq \tilde{\ell}$ and consider an arbitrary box $s \in \mathcal{N}(t)$. Then~$t$ and $s$ violate \eqref{ow:eq_adm_cone_dist}, which means that $\eta_2 \dist{t}{s} < \kappa q_\ell^2$, i.e.~there exist $x \in \bar{t}$ and $y \in \bar{s}$ such that $|x-y| < \kappa q_\ell^2 / \eta_2$. Hence, we can estimate
\begin{equation} \label{ow:eq_est_r_ell}
\begin{split}
	\max_{z \in \bar{s}}|z-m_t| &\leq \max_{z \in \bar{s}}\left(|z-y|+|x-y|+|x-m_t|\right) \\
	&\leq 	q_\ell + \frac{\kappa q_\ell^2}{\eta_2} + \frac{q_\ell}{2} 
	\leq \left(\frac{3}{2} + \frac{1}{\eta_2}\right) \kappa q_\ell^2 =: r_\ell,
\end{split}
\end{equation}
where we used $\kappa q_\ell > 1$ for the last estimate. Therefore, every box $s \in \mathcal{N}(t)$ is contained in the ball $B_{r_\ell}(m_t)$ with $r_\ell$ from \eqref{ow:eq_est_r_ell}. If instead $\ell > \tilde{\ell}$ we analogously show
\begin{equation} \label{ow:eq_est_r_ell_lf}
	\mathcal{N}(t) \subset B_{r_\ell}(m_t), \quad 
	r_\ell = \left(\frac{3}{2} + \frac{1}{\eta_2}\right) q_\ell.
\end{equation}
With the ball $B_{r_\ell}(m_t)$ covering $\mathcal{N}(t)$ we can estimate
\begin{equation} \label{ow:eq_est_nearfield}
	\# \mathcal{N}(t) \leq \frac{|B_{r_{\ell}}(m_t)|}{v_{\ell}(\mathcal{T}_S)} 
	= \frac{(4 \pi/ 3) r_{\ell}^3}{3^{-3/2}q_{\ell}(\mathcal{T}_S)^3} 
	= 4 \pi \sqrt{3} \left(\frac{r_{\ell}}{q_{\ell}(\mathcal{T}_S)}\right)^3,
\end{equation}
where $v_{\ell}(\mathcal{T}_S) = 3^{-3/2} q_{\ell}(\mathcal{T}_S)^3$ denotes the volume of boxes $s \in \mathcal{T}_S^{(\ell)}$. Since $t \in \mathcal{T}_T^{(\ell)}$ was arbitrary, the bound in \eqref{ow:eq_est_nearfield} holds also for $N_{\mathcal{N},T}^{(\ell)}$ instead of $\mathcal{N}(t)$.

Summarizing \eqref{ow:eq_est_n_c_ell} and above findings, we get for the number $N_\mathrm{C}$ of all coupling matrices the estimate
\begin{align} 
	N_\mathrm{C} &= \sum_{\ell = 1}^{p_{\min}} N_{\mathrm{C}}^{(\ell)} \leq 
	\sum_{\ell = 1}^{p_{\min}} 8^{\ell+1} N_{\mathcal{N},T}^{(\ell-1)} \leq 
	\sum_{\ell = 1}^{p_{\min}} 8^{\ell+1} 4 \pi \sqrt{3} 
		\left(\frac{r_{\ell - 1}}{q_{\ell-1}(\mathcal{T}_S)}\right)^3\nonumber \\
	&\leq 4 \pi \sqrt{3} \left(\frac{3}{2} + \frac{1}{\eta_2}\right)^3 
	\left( \sum_{\ell = 1}^{\tilde{\ell}+1}  8^{\ell+1} (\kappa c_{\text{un}} q_\ell)^3 + 
	\sum_{\ell = \tilde{\ell} + 2}^{p_{\min}} 8^{\ell+1} c_{\text{un}}^3 \right) \nonumber\\
	&\leq c_\mathrm{C} \left(\sum_{\ell=1}^{\tilde{\ell}+1} (\kappa q_0)^3 + 8^{p_{\min}} \right)
	\leq c_\mathrm{C} (p_{\min} (\kappa q_0)^3 + 8^{p_{\min}}),\label{ow:eq_est_n_coupling}
\end{align}
where we assumed that $1 \leq \tilde{\ell} + 1 < p_{\min}$ and used assumption~\eqref{ow:eq_assume_comparable_trees} and the relation ${q_0 = 2^\ell q_\ell}$. If either $\tilde{\ell} + 1 \geq p_{\min}$ or $\tilde{\ell} = -1$, one can repeat the estimates in \eqref{ow:eq_est_n_coupling} and ends up with a similar result where one can cancel $8^{p_{\min}}$ in the first case and $p_{\min}(\kappa q_0)^3$ in the second case. The assertions about the complexity follow directly from \eqref{ow:eq_est_n_coupling} with assumptions \eqref{ow:eq_assume_tree_uniformity} and \eqref{ow:eq_assume_resolve_wavelength} or \eqref{ow:eq_assume_overresolve_wavelength}, respectively, since every coupling matrix~\eqref{ow:eq_def_coupling_matrix} has $(m+1)^6 = \mathcal{O}(1)$ entries. \qed
\end{proof}

In Thm.~\ref{ow:thm_nr_coupling_mat} we have estimated the number $N_\mathrm{C}$ of all coupling matrices, which corresponds to the number of admissible blocks $\mathcal{L}_{T \times S}^+$. As explained in Sect.~\ref{ow:sec_implementation_details}, we store reoccuring matrices only once to reduce the related storage costs drastically as we will see in the next theorem and in Sect.~\ref{ow:sec:num:examples}. Since one needs to know all blocks in $\mathcal{L}_{T \times S}^+$ in Alg.~\ref{ow:alg_fdmvm} and storing them has complexity $\mathcal{O}(N_\mathrm{C})$, storing each matrix only once does not reduce the overall storage complexity of the method asymptotically. 

\begin{theorem} \label{ow:thm_stored_coupling_mat}
Let the root boxes $T$ and $S$ of $\mathcal{T}_T$ and $\mathcal{T}_S$ be identical up to translation. Then the number $N_{\mathrm{SC}}$ of coupling matrices $A_{c,t \times s}$ which have to be stored can be estimated by
\begin{equation} \label{ow:eq_estimate_max_coup_mat_stored}
	N_{\mathrm{SC}}\leq p_{\min} \max\{ c_{\mathrm{C}},
	c_{\mathrm{C}}^{2/3}(\kappa q_0)^{2}\}.
\end{equation}
If \eqref{ow:eq_assume_tree_uniformity} and \eqref{ow:eq_assume_resolve_wavelength} hold, the corresponding storage complexity is $\mathcal{O}(N^{2/3} \log(N))$.
\end{theorem}

\begin{proof}
From the proof of Thm.~\ref{ow:thm_nr_coupling_mat}, in particular \eqref{ow:eq_est_r_ell}, \eqref{ow:eq_est_r_ell_lf} and \eqref{ow:eq_est_nearfield}, it follows that the number of admissible blocks $(t,s) \in \mathcal{L}_{T \times S}^+$ for a fixed box $t \in \mathcal{T}_T^{\ell}$  can be estimated by
\begin{equation} \label{ow:eq_est_nr_adm_blocks_gen}
\begin{split}
	8 \cdot \#\mathcal{N}(\parent(t)) &\leq 32 \pi \sqrt{3} c_\mathrm{un}^3 \left( \frac{3}{2} + \frac{1}{\eta_2}\right)^3 \max\{1, (\kappa q_0)^3 8^{1-\ell}\} \\
	&\leq c_{\mathrm{C}} \max\{1, (\kappa q_0)^3 8^{-\ell}\},
\end{split}
\end{equation}
where we used $q_\ell = 2^{-\ell} q_0$, and $c_{\mathrm{C}}$ is the same constant as in \eqref{ow:eq_est_nr_coupling_mat}. For a different box $t' \in \mathcal{T}_T^{\ell}$ the boxes $s'$ such that $(t',s') \in \mathcal{L}_{T \times S}^+$ are identical to blocks $(t,s)$ up to translation, which follows from the assumption on the root boxes $T$ and $S$ and the uniformity of the trees~$\mathcal{T}_T$ and $\mathcal{T}_S$. Hence, the coupling matrices coincide and~\eqref{ow:eq_est_nr_adm_blocks_gen} is a bound for the number $N_{\mathrm{SC}}^{(\ell)}$ of stored coupling matrices at level $\ell$. On the other hand, there are at most $8^{2 \ell}$ blocks at level $\ell$ of $\mathcal{T}_{T \times S}^+$, which gives
\begin{equation*}
	N_{\mathrm{SC}}^{(\ell)} \leq 
	\min\{8^{2\ell}, c_{\mathrm{C}} \max\{1, (\kappa q_0)^3 8^{-\ell} \} \}.
\end{equation*}
The maximum over all $\ell$ of the expression on the right-hand side is bounded by~$8^{2 \ell^*}$, where $\ell^*$ is the intersection point of $8^{2 \ell}$ and $c_{\mathrm{C}} \max\{1, (\kappa q_0)^3 8^{-\ell} \}$. By computing this maximum we end up with the general bound
\begin{equation*}
	N_{\mathrm{SC}}^{(\ell)} \leq 
	\max\left\{ c_\mathrm{C}, c_{\mathrm{C}}^{2/3} (\kappa q_0)^2 \right\} \quad \text{for all }\ell \geq 0.
\end{equation*}
Summation over all levels $\ell = 1, \ldots, p_{\min}$ yields \eqref{ow:eq_estimate_max_coup_mat_stored}. Since every coupling matrix has $(m+1)^6$ entries, it follows that all distinct coupling matrices can be stored with $\mathcal{O}(N^{2/3}\log(N))$ memory units, if assumptions \eqref{ow:eq_assume_tree_uniformity} and~\eqref{ow:eq_assume_resolve_wavelength} hold. \qed
\end{proof}

In Thm.~\ref{ow:thm_est_nearfield_entries}, we will perform the complexity analysis of the nearfield evaluation, i.e.~lines \ref{ow:alg_line_start_nearfield} and \ref{ow:alg_line_end_nearfield} of Alg.~\ref{ow:alg_fdmvm}.
In unbalanced trees there can be leaf clusters at coarse levels with large nearfields. If there were many of these, the complexity would not be linear. To exclude exceptional settings we make the additional assumption that the number of such leaf clusters is bounded, i.e., there exists a constant $c_{\mathrm{in}} \in \mathbb{N}$ such that
\begin{equation} \label{ow:eq_assume_few_uncontrollable_leaves}
  \# \mathcal{L}_T^-\leq c_\mathrm{in}, \quad \# \mathcal{L}_S^- \leq c_\mathrm{in},
\end{equation}
where $\mathcal{L}_S^-:= \mathcal{L}_S \setminus \mathcal{L}_S^+$ and
\begin{equation} \label{ow:eq_def_nice_leaves}
\begin{split}
	\mathcal{L}_{S}^+ := 
	\{s \in \mathcal{L}_S \colon &\# \{t: (t,s) \in \mathcal{L}_{T\times S}^-\} \leq c_\mathrm{nf} \text{ and } \\ 
	&\#\hat{t} \leq c_{\mathrm{m}} n_{\max} \text{ for all } (t,s) \in \mathcal{L}_{T \times S}^- \},
\end{split}
\end{equation}
for some fixed parameters $c_\mathrm{m}$ and $c_\mathrm{nf}$.
It follows from~\eqref{ow:eq_est_nr_adm_blocks_gen} that for a leaf box $s \in \mathcal{T}_S^{(\ell)}$ the assumption $\# \{t: (t,s) \in \mathcal{L}_{T\times S}^-\} \leq c_\mathrm{nf}$ holds true for sufficiently large constant $c_\mathrm{nf}$ if $\ell = \level(t)$ is large enough.

\begin{theorem} \label{ow:thm_est_nearfield_entries}
Let assumptions \eqref{ow:eq_assume_comparable_trees} and \eqref{ow:eq_assume_few_uncontrollable_leaves} hold true.
Then there exists a constant~$c_\mathrm{D}$ depending only on $c_{\mathrm{un}}$, $c_\mathrm{in}$, $c_\mathrm{m}$, $c_\mathrm{nf}$, $n_{\max}$ and $\eta_2$ such that the number~$M_\mathrm{D}$ of entries of all nearfield blocks $A|_{\hat{t} \times \hat{s}}$ in Alg.~\ref{ow:alg_fdmvm} is bounded by
\begin{equation} \label{ow:eq_est_nearfield_mat_entries}
	M_\mathrm{D} \leq c_\mathrm{D} (N_T + N_S + 8^{p_{\min}}).
\end{equation}
If \eqref{ow:eq_assume_tree_uniformity} holds, the corresponding storage complexity is $\mathcal{O}(N)$.
\end{theorem}

\begin{proof}
Each nearfield matrix block corresponds to an inadmissible block $(t,s) \in \mathcal{L}_{T \times S}^-$. For such a block there holds $t \in \mathcal{L}_T$ or $s \in \mathcal{L}_S$ by construction. We start counting entries of blocks corresponding to leaves in $\mathcal{L}_S$ by considering the sets $\mathcal{L}_S^+$ and~$\mathcal{L}_S^-$.

For the number $M_{\mathrm{D},S}^-$ of nearfield matrix entries corresponding to blocks~$(t,s)$ with outlying leaves $s \in \mathcal{L}_S^-$ there holds
\begin{equation} \label{ow:eq_est_uncontrollable}
	M_{\mathrm{D},S}^- 
	= \sum_{s \in \mathcal{L}_{S}^-} \# \hat{s} 
		\sum_{\{t: (t,s) \in \mathcal{L}_{T \times S}^-\}} \# \hat{t}
	\leq c_{\mathrm{in}} n_{\max} N_T.
\end{equation}
Here we used that $\#\hat{s} \leq n_{\max}$ holds for all leaf boxes, and that the nearfield $\mathcal{N}(s) =\{t: (t,s) \in \mathcal{L}_{T \times S}^-\}$ of $s$ can contain at most all $N_T$ points in $P_T$.

Next we estimate the number $M_{\mathrm{D},S}^+$ of nearfield matrix entries corresponding to blocks $(t,s)$ with $s \in \mathcal{L}_S^+$. For  fixed $s \in \mathcal{L}_S^+$ there exist at most $c_\mathrm{nf}$ such blocks $(t,s)$ and the corresponding boxes $t$ contain maximally $c_\mathrm{m} n_{\max}$ points by definition of $\mathcal{L}_S^+$ in \eqref{ow:eq_def_nice_leaves}. Furthermore, the level of a box $s$ in an inadmissible block $(t,s)$ can be at most $p_{\min}$ and $\mathcal{T}_S$ can have at most $8^{p_{\min}}$ leaves at levels $\ell \leq p_{\min}$. Hence, we get
\begin{equation} \label{ow:eq_est_controllable}
	M_{\mathrm{D},S}^+ 
	= \sum_{s \in \mathcal{L}_S^+} \# \hat{s} \sum_{t \in \mathcal{N}(s)} \# \hat{t} 
	\leq 8^{p_{\min}} c_\mathrm{nf} c_\mathrm{m} n_{\max}^2.
\end{equation}

Analogous estimates as \eqref{ow:eq_est_uncontrollable} and \eqref{ow:eq_est_controllable} hold true for nearfield matrices corresponding to leaves in $\mathcal{L}_T$. Adding up  all these estimates leads to the bound in~\eqref{ow:eq_est_nearfield_mat_entries}, with the constant $c_{\mathrm{D}} = 2 n_{\max}\max\{c_\mathrm{in}, c_\mathrm{nf} c_\mathrm{m} n_{\max}\}$. If \eqref{ow:eq_assume_tree_uniformity} holds, the storage complexity $\mathcal{O}(N)$ is an immediate consequence of \eqref{ow:eq_est_nearfield_mat_entries}. \qed
\end{proof}
The following theorem summarizes the results of this section.
\begin{theorem}
Let assumptions \eqref{ow:eq_assume_sparse_leafs}--\eqref{ow:eq_assume_reasonable_lhf} and \eqref{ow:eq_assume_few_uncontrollable_leaves} hold true. Then the complexity of Alg.~\ref{ow:alg_fdmvm} is $\mathcal{O}(N \log(N))$. If \eqref{ow:eq_assume_resolve_wavelength} is replaced by \eqref{ow:eq_assume_overresolve_wavelength} the complexity is reduced to $\mathcal{O}(N)$.
\end{theorem}
\section{Numerical Examples}\label{ow:sec:num:examples}
In this section we want to test the method presented in Sect.~\ref{ow:sec_derivation} and to validate the theoretical results from Sect.~\ref{ow:sec_complexity_analysis}. For this purpose we use a single core implementation of Alg.~\ref{ow:alg_fdmvm} in C\texttt{++} on a computer with \new{384 GiB RAM} and 2 Intel Xeon Gold 5218 CPUs. To reduce the required memory we store only the non-directional parts $E_{t'}$ of the transfer matrices and each coupling matrix once, as described in Sect.~\ref{ow:sec_implementation_details}. However, if the matrix is applied several times it can be beneficial to store also the directional interpolation matrices $L_{t,c}$ and nearfield matrix blocks.

For the tests we consider points distributed uniformly inside the cube~$[-1,1]^3$. For various values $k \geq 3$ we choose $\tilde{x}_n = (2n-1) 2^{-k} -1$ in~$[-1,1]$ for all $n \in \{1, ..., 2^k\}$ and construct the set of points $P_T(k) = \{x_j\}_{j=1}^{N(k)}$ with ${N(k) = 8^k}$ as tensor products of these one-dimensional points. We choose $P_S(k)= P_T(k)$ and consider the matrix $A$ as in \eqref{ow:eq_helmholtz_matrix} with the wave number $\kappa = 0.1 \cdot 2^k$ and the diagonal set to zero to eliminate the singularities. The approximation derived in Sect.~\ref{ow:sec_derivation} is applicable despite the change of the diagonal because it effects only parts of the matrix which are evaluated directly. 

We construct a uniform box cluster tree $\mathcal{T}_T$ for the set $P_T$ using Alg.~\ref{ow:alg_box_ct} with the initial box $T=[-1,1]^3$ and the parameter $n_{\max} = 512$. With this choice of parameters and points, $\mathcal{T}_T$ is a uniform octree with depth $p(\mathcal{T}_T)=k-3$, where every leaf contains exactly 512 points. We construct the sets of directions~$D^{(\ell)}$  with Alg.~\ref{ow:alg_dir} and the largest high frequency level $\ell_{\mathrm{hf}} = k-4$ and finally we use Alg.~\ref{ow:alg_block_ct} to construct the block cluster tree $\mathcal{T}_{T \times T}$ with the parameter $\eta_2=5$. The parameters $\ell_\mathrm{hf}$ and $\eta_2$ were chosen according to the parameter choice rule in \cite[Sect.~3.1.4]{ow:watschinger}. In particular, the choice $\eta_2=5$ minimizes the number of inadmissible blocks $b \in \mathcal{L}^-_{T \times T}$ at levels $\ell > \ell_\mathrm{hf}$. Note that due to the uniformity of the tree $\mathcal{T}_T$ and the choice $P_S(k)=P_T(k)$ the block tree $\mathcal{T}_{T \times T}$ has depth $p(\mathcal{T}_{T})=\ell_\mathrm{hf} + 1$ and all inadmissible blocks are at level $\ell_\mathrm{hf}+1$. 

The assumptions \eqref{ow:eq_assume_sparse_leafs}--\eqref{ow:eq_assume_reasonable_lhf} are all satisfied for the considered examples for suitable constants $c_{\max}$, $c_\mathrm{un}$, $c_\mathrm{ad}$, $c_\mathrm{geo}$, and $c_\mathrm{hf}$ independent of the sets $P_T(k)$. Assumption \eqref{ow:eq_assume_few_uncontrollable_leaves} holds for $c_\mathrm{in}=0$, because all leafs in $\mathcal{L}_T$ are at level $k-3$ and by the choice of $\eta_2$ there holds $\#\{s:(t,s) \in \mathcal{L}_{T \times T}^-\} \leq 27$ for all leaves $t \in \mathcal{L}_T$.

\begin{table}[t]
\caption{Computation times and storage requirements for matrix-vector multiplications using Alg.~\ref{ow:alg_fdmvm} for the matrix $A$ corresponding to sets of points $P_T(k)$ for various values of $k$. Parameters: $\ell_\mathrm{hf}=k-4$, $\eta_2=5$, interpolation degree $m=4$.} \label{ow:tab_times_and_storage}
\centering
\begin{tabular}{l |  r  r  r  r  r  r  r  r  r  r}
$k$ & $N$ & $\kappa$ & $t_\mathrm{tot}$ & $t_\mathrm{s}$ & $t_\mathrm{nf}$ & $t_\mathrm{ff}$ & nf [\%] & $N_\mathrm{SC}$ & $N_\mathrm{C}$ & [GiB] \\
\hline
5 & 32768 & 3.2 & 7.31 & 0.34 & 6.94 & 0.03 & 24.41 & 316 & 3096 & 0.02\\
6 & 262144 & 6.4 & 76.25 & 1.20 & 74.29 & 0.76 & 4.06 & 1522 & 166320 & 0.10\\
7 & 2097152 & 12.8 & 702.72 & 3.71 & 688.14 & 10.87 & 0.58 & 4554 & 2640960 & \new{0.46} \\
8 & 16777216 & 25.6 & 6060.16 & 15.24 & 5907.66 & 137.26 & 0.077 & 9824 & 33103296 & \new{3.09} \\
9 & 134217728 & 51.2 & 50204.00 & 118.89 & 48576.20 & 1508.91 & 0.010 & 32036 & 344979432 & 24.2\\	
\end{tabular}
\end{table}
\tikzexternalenable
\begin{figure}[ht]
	\minipage{0.549\textwidth}
        \includegraphics{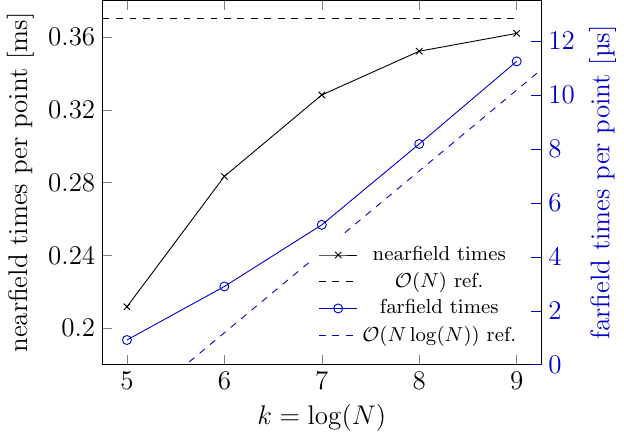}
	\endminipage\hfill
	\minipage{0.449\textwidth}
        \includegraphics{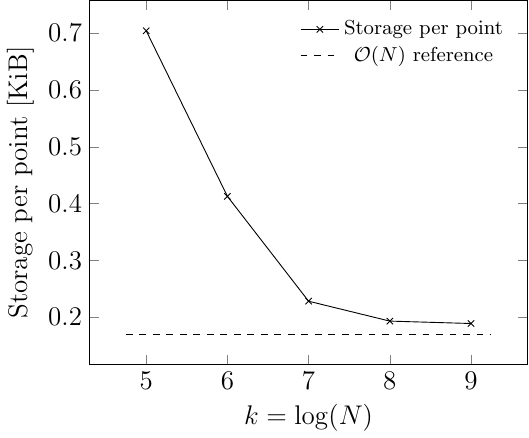}
	\endminipage\hfill
	\caption{ \new{ Plots related to the computations of Table~\ref{ow:tab_times_and_storage}. Left image: Nearfield and farfield computation times per point with linear and quasi-linear reference curves in the different scales. Right image: Required storage per point with linear reference curve.} } \label{ow:fig_times_and_storage}
      \end{figure}
       \tikzexternaldisable

In the described setting we apply Alg.~\ref{ow:alg_fdmvm} for the fast multiplication of the matrix $A$ with a randomly constructed vector $v$. The interpolation degree $m=4$ is chosen, since it is reasonably high to yield a good approximation quality (e.g.~relative error $2\cdot 10^{-4}$ for $k=6$) while it is low enough to make the approximations of all admissible blocks efficient.

The results of the computations for various sets of points $P_T(k)$ are given in Table~\ref{ow:tab_times_and_storage} and \new{Fig.}~\ref{ow:fig_times_and_storage}.  The total computational times~$t_{\mathrm{tot}}$ are split into setup times~$t_\mathrm{s}$,  times~$t_\mathrm{nf}$ of the nearfield part, and computational times~$t_\mathrm{ff}$ of the farfield part.
In addition, the percentage of matrix entries in inadmissible blocks~(nf), the numbers $N_{\mathrm{SC}}$ and $N_{\mathrm{C}}$ of stored and applied coupling matrices and the storage requirements ([GiB]) are given.
\new{A direct computation for $k=7$ takes more than 32 hours. Thus the directional approximation is about 160 times faster.
  For larger examples the difference would be even more pronounced due to the quadratic complexity of the direct computation.
  
In Fig.~\ref{ow:fig_times_and_storage}, we plot computational times and memory consumption per point. As expected from our theoretical results of Sect.~\ref{ow:sec_complexity_analysis}, we observe linear and almost linear behavior, respectively, for the nearfield and the farfield part of the computations, see the left plot in Fig.~\ref{ow:fig_times_and_storage}. As usual there is some preasymptotic behavior in such plots. The right plot in Fig.~\ref{ow:fig_times_and_storage} shows the linear behavior of the memory requirements. Note that we store coupling matrices and transfer matrices only. In particular, we mention the low number $N_{\mathrm{SC}}$ of stored coupling matrices compared to the total number~$N_{\mathrm{C}}$ of coupling matrices in Table~\ref{ow:tab_times_and_storage}.
}

\noindent \textbf{Acknowledgment.}
This work was partially supported by the Austrian Science Fund (FWF): I 4033-N32.

 
 \end{document}